\documentclass[11pt]{article}
\usepackage{latexsym}

\font\german = eufm10 scaled\magstep1
\font\Cp = msbm10

\newcommand{\Ccc}{\hbox{\Cp C}}
\newcommand{\Nnn}{\hbox{\Cp N}}
\newcommand{\Ppp}{\hbox{\Cp P}}

\newcommand{\Zzz}{\hbox{\Cp Z}}

\newcommand{\Ssss}{\hbox{\german S}}

\newcommand{\type}{\hbox{\rm type}}

\newcommand{\qed}{\mbox{$\Box$}\vspace{\baselineskip}}

\newenvironment{proof}{\noindent {\bf Proof:}}
                      {{\qed}}
\newenvironment{proof_special}{\noindent {\bf Proof:}}
                              {\vspace{-2mm}}

\newtheorem{theorem}{Theorem}[section]
\newtheorem{proposition}[theorem]{Proposition}
\newtheorem{lemma}[theorem]{Lemma}
\newtheorem{definition}[theorem]{Definition}
\newtheorem{corollary}[theorem]{Corollary}
\newtheorem{example}[theorem]{Example}

\newcommand{\hz}{\hat{0}}
\newcommand{\ho}{\hat{1}}

\newcommand{\ab}{\av\bv}
\newcommand{\av}{{\bf a}}
\newcommand{\bv}{{\bf b}}

\newcommand{\Des}[1]{\beta\left(#1\right)}
\newcommand{\Desq}[1]{\beta_{q}\!\left(#1\right)}
\newcommand{\qb}[2]{{\left[{#1 \atop #2} \right]}}

\newcommand{\sech}{\mbox{\rm sech}}

\begin{document}

\title{Exponential Dowling Structures\thanks{European
Journal of Combinatorics {\bf 30} (2009), 311--326.}}

\author{{\sc Richard EHRENBORG and Margaret A.\ READDY}}

\date{}

\maketitle


\begin{abstract}
The notion of exponential Dowling structures
is introduced, generalizing Stanley's original
theory of exponential structures.
Enumerative theory is developed to determine the
M\"obius function of exponential Dowling structures,
including a restriction
of these structures
to elements whose types satisfy a semigroup condition.
Stanley's study of permutations associated with
exponential structures leads to a similar vein of study for
exponential Dowling structures.
In particular, for the extended $r$-divisible partition lattice
we show the M\"obius function is, up to a sign, the number of permutations
in the symmetric group on $rn+k$ elements having descent
set $\{r, 2r, \ldots, nr\}$.
Using Wachs' original $EL$-labeling of the $r$-divisible partition lattice,
the extended $r$-divisible partition lattice is shown to be
$EL$-shellable.  
\end{abstract}

\section{Introduction}
\label{section_introduction}
\setcounter{equation}{0}

Stanley introduced 
the notion of exponential structures,
that is,
a family of posets that have
the partition lattice $\Pi_n$ as the
archetype~\cite{Stanley_e_s, Stanley_EC_II}.
His original motivation was 
to explain certain permutation phenomena.
His theory ended up inspiring 
many mathematicians
to study the partition lattice and other exponential structures
from enumerative, representation theoretic and homological perspectives.

For example, Stanley studied the
$r$-divisible partition lattice $\Pi_{n}^{r}$
and computed its
M\"obius numbers~\cite{Stanley_e_s}.
Calderbank, Hanlon and Robinson~\cite{Calderbank_Hanlon_Robinson}
derived plethystic formulas in order
to determine
the character of the representation of the symmetric group on its top
homology, while
Wachs determined the homotopy type,
gave explicit bases for the homology and cohomology
and 
studied the $\Ssss_{n}$ action on the top homology~\cite{Wachs_1}.
For the poset of partitions
with block sizes divisible by~$r$ and having
cardinality at least $rk$, 
a similar array of questions have been considered
by 
Bj\"orner and Wachs, Browdy, 
Linusson, Sundaram  and 
Wachs~\cite{Bjorner_Wachs_nonpure, Browdy,
Linusson, Sundaram_applications_Hopf,
Wachs_2}.
Other related work can be found
in~\cite{Bjorner_Lovasz, Bjorner_Welker, 
Gottlieb_Wachs, Sundaram_homology, Sundaram_Wachs, Welker},
as well as work of
Sagan~\cite{Sagan}, who showed certain examples of exponential structures
are $CL$-shellable.

In this paper
we extend Stanley's notion of exponential structures
to that 
of  {\em exponential Dowling 
structures}.
The prototypical example is
the Dowling lattice~\cite{Dowling}.
It can most easily be viewed
as the intersection lattice of the
complex hyperplane arrangement
in~(\ref{equation_complex_hyperplane_arrangement}).
See Section~\ref{section_Dowling_lattice} for a review 
of the Dowling
lattice.

In Section~\ref{section_Dowling_exponential_structures}
we introduce exponential Dowling structures.
We derive the compositional formula 
for exponential Dowling structures
analogous 
to Stanley's theorem
on the compositional formula
for exponential structures~\cite{Stanley_e_s}.
As an application, we give the generating function for the
M\"obius  numbers of an exponential Dowling structure.

An important method to generate new exponential Dowling structures
from old ones is given in
Example~\ref{example_r_k}.
Loosely speaking,
in this new structure
an
$r$-divisibility condition holds for the ``non-zero blocks''
and the cardinality of the ``zero block'' satisfies
the more general condition of being greater
than or equal to $k$ and congruent to $k$ modulo $r$.
We will return to many important special cases of this example
in later sections.

In Section~\ref{section_Mobius_function}
we consider restricted forms
of both exponential and
exponential Dowling structures.
In the case the exponential Dowling structure
is restricted to elements whose
type satisfies a semigroup condition,
the generating function for the M\"obius function of this poset
is particularly elegant.
See Corollary~\ref{corollary_semigroup}
and Proposition~\ref{proposition_M}.
When the blocks have even size, the
generating function is nicely expressed
in terms of the hyperbolic functions.
See Corollary~\ref{corollary_hyperbolic}.

In Section~\ref{section_permutations}
we continue to develop the connection
between permutations and structures
first studied by Stanley
in the case of exponential structures.
In particular we consider 
the lattice
$\Pi_{m}^{r,j}$,
an extension
of the
$r$-divisible partition lattice $\Pi_{m}^{r}$.
In Section~\ref{section_EL} we verify that
Wachs' $EL$-labeling of the $r$-divisible
partition lattice $\Pi_{m}^{r}$ naturally extends to
the new lattice $\Pi_{m}^{r,j}$.

We end with remarks and open questions regarding
further exponential Dowling structures and their connections
with permutation statistics.

\section{The Dowling lattice}
\label{section_Dowling_lattice}
\setcounter{equation}{0}

Let $G$ be a finite group of order $s$.
The {\em Dowling lattice} $L_{n}(G) = L_{n}$
has the following combinatorial description.
For the original formulation, 
see Dowling's paper~\cite{Dowling}.
Define an {\em enriched block} $\widetilde{B} = (B,f)$ to be
a non-empty subset $B$ of $\{1,\ldots,n\}$ and
a function $f : B \longrightarrow G$.
Two enriched blocks $\widetilde{B} = (B,f)$ and
$\widetilde{C} = (C,g)$ are said to be equivalent if
$B = C$ and the functions $f$ and $g$ differ only by a multiplicative
scalar, that is,
there exists $\alpha \in G$ such that
$f(b) =  g(b) \cdot \alpha$ for all $b$ in $B$.
Hence there are only $s^{|B| - 1}$ possible ways to enrich a
non-empty set~$B$, up to equivalence.
Let $\widetilde{B} = (B,f)$ and $\widetilde{C} = (C,g)$
be two disjoint enriched blocks and
let~$\alpha$ be an element in~$G$. 
We can define a function~$h$ on the block $B \cup C$
by
$$ h(b) = \left\{ \begin{array}{c c l}
       f(b)   & \mbox{ if } & b \in B, \\
     \alpha \cdot g(b) & \mbox{ if } & b \in C.
                  \end{array} \right.  $$
Since the group element
$\alpha$ can be chosen in $s$ possible ways,
there are $s$ possible
ways to merge two enriched blocks.

For $E$ a subset of $\{1, \ldots, n\}$,
an {\em  enriched partition}
$\widetilde{\pi} = \{\widetilde{B}_{1}, \ldots, \widetilde{B}_{m}\}$
on the set $E$
is a partition $\pi = \{B_{1}, \ldots, B_{m}\}$ of $E$, 
where each block $B_i$
is enriched with a function $f_{i}$.
The elements of the Dowling lattice~$L_{n}$
are the collection
$$   L_{n}
   =
     \left\{ (\widetilde{\pi}, Z) 
            \:\: : \:\:  Z \subseteq \{1, \ldots, n\}
			\mbox{ and }
                       \widetilde{\pi} 
                      \mbox{ is an enriched partition of } 
 \overline{Z} = \{1, \ldots, n\} - Z  \right\}   .  $$
The set $Z$ is called the {\em zero block.}
Define the cover relation on
$L_{n}$ by the following two relations:
$$  \begin{array}{r c l}
 (\{\widetilde{B}_1, \widetilde{B}_2, \ldots, \widetilde{B}_m\}, Z)
    & \prec &
  (\{\widetilde{B}_2, \ldots, \widetilde{B}_m\}, Z \cup B_1) , \\
 (\{\widetilde{B}_1, \widetilde{B}_2, \ldots, \widetilde{B}_m\}, Z)
    & \prec &
  (\{\widetilde{B}_1 \cup \widetilde{B}_2, \ldots, \widetilde{B}_m\}, Z) .
    \end{array} $$
The first relation says that a block is allowed 
to merge with the zero set.
The second relation says that two blocks are allowed
to be merged together.
The minimal element $\hz$ corresponds to the partition
having all singleton blocks and empty zero block,
while the maximal element $\ho$ corresponds to the partition
where all the elements lie in the zero block.
Observe that the Dowling lattice $L_{n}$ is
graded of rank $n$.

When the group $G$ is the cyclic group of order $s$,
that is, $\Zzz_{s}$,
the Dowling lattice has the following geometric description.
Let $\zeta$ be a
primitive $s$th root of unity.
The {\em Dowling lattice}
$L_{n}(\Zzz_{s})$ is the intersection lattice of
the complex hyperplane arrangement
\begin{equation}
   \left\{
   \begin{array}{c c c l}
  z_i & = & \zeta^{h} \cdot z_j &
        \mbox{ for } 1 \leq i < j \leq n
        \mbox{ and } 0 \leq h \leq s-1, \\
  z_i          & = & 0 &   \mbox{ for } 1 \leq i \leq n ,
   \end{array} 
   \right. 
\label{equation_complex_hyperplane_arrangement}
\end{equation}
that is, the collection of all possible intersections
of these hyperplanes ordered by reverse inclusion.

In the notation we will suppress
the Dowling lattice's
dependency on the group $G$.
Only the order~$s$ of the group will
matter in this paper.
In Section~\ref{section_permutations}
the order $s$ will be specialized to the value $1$.

For an element $x = (\widetilde{\pi}, Z)$
in the Dowling lattice $L_{n}$,
define the {\em type} of $x$ to be
$(b; a_1, a_2, \ldots, a_n)$,
where
$a_i$ is the number of blocks in
$\widetilde{\pi}$
of size $i$ in $x$ and
$b$ is the size of the zero block $Z$.
Observe
that 
the interval $[x,\ho]$ in the Dowling lattice
is isomorphic to $L_{n-\rho(x)}$
where $\rho$ denotes the rank function.
Moreover, 
the interval
$[\hz,x]$
is isomorphic to
$L_b \times \Pi_1^{a_1} \times \cdots \times \Pi_n^{a_n}$,
where
$(b; a_1, a_2, \ldots, a_n)$ is the type of $x$
and
$\Pi_j^{a_j}$ denotes the Cartesian product of
$a_j$ copies of the partition lattice on $j$ elements.

\begin{lemma}
In the Dowling lattice
$L_n$ there are
$$\frac{s^n \cdot n!}
       {s^b \cdot b! \cdot (s \cdot 1!)^{a_1} \cdot a_1! 
        \cdot (s \cdot 2!)^{a_2} \cdot a_2! 
        \cdots (s \cdot n!)^{a_n} \cdot a_n!}
$$
elements of type
$(b; a_1, a_2, \ldots, a_n)$.
\label{lemma_count_type}
\end{lemma}
\begin{proof}
For an element
of type $(b;a_1, a_2, \ldots, a_n)$
in the Dowling lattice $L_n$
we can
choose the $b$ elements in the zero-set in ${n \choose b}$ ways.
The underlying partition on the remaining $n-b$ elements can be
chosen in
$$   \displaystyle
{n-b \choose 
\underbrace{1, \ldots, 1}_{a_1}, 
\underbrace{2, \ldots, 2}_{a_2},
\ldots, 
\underbrace{n, \ldots, n}_{a_n}}
\cdot
\frac{1}{a_1 ! \cdot a_2! \cdots a_n!}      $$
ways.
For a block of size $k$
there are $s^{k-1}$ signings,
so the result follows. 
\end{proof}

\section{Dowling exponential structures}
\label{section_Dowling_exponential_structures}
\setcounter{equation}{0}

Stanley introduced the notion  of an exponential 
structure.  See~\cite{Stanley_e_s} and~\cite[Section 5.5]{Stanley_EC_II}.

\begin{definition}
An {\em exponential structure}
${\bf Q} = (Q_{1}, Q_{2}, \ldots)$ 
is a sequence of posets
such that
\begin{itemize}
\item[(E1)]
The poset $Q_{n}$ has a unique maximal element $\ho$
and every maximal chain in $Q_{n}$ contains
$n$ elements.

\item[(E2)]
For an element $x$ in $Q_n$ of rank $k$,
the interval $[x,\ho]$ is isomorphic to the
partition lattice on $n-k$ elements, $\Pi_{n-k}$.

\item[(E3)]
The lower order ideal generated by $x \in Q_n$
is isomorphic to
$Q_1^{a_1} \times \cdots \times Q_n^{a_n}$.
We call $(a_1, \ldots, a_n)$ the {\em type}
of $x$.

\item[(E4)]
The poset $Q_n$ has $M(n)$ minimal elements.
The sequence $(M(1), M(2), \ldots)$ is called the
{\em denominator sequence}.

\end{itemize}
\end{definition}

Analogous to the definition of an exponential structure,
we introduce
the notion of an exponential Dowling structure.

\begin{definition}
An {\em exponential Dowling structure}
${\bf R} = (R_{0}, R_{1}, \ldots)$ 
associated to an exponential structure 
${\bf Q} = (Q_1, Q_2, \ldots)$
is a sequence of posets
such that
\begin{itemize}
\item[(D1)]
The poset $R_{n}$ has a unique maximal element $\ho$
and every maximal chain in $R_{n}$ contains
$n+1$ elements.

\item[(D2)]
For an element $x \in R_n$,
$[x,\ho] \cong L_{n - \rho(x)}$.

\item[(D3)]
Each element $x$ in $R_{n}$
has a {\em type} $(b; a_1, \ldots, a_n)$ assigned
such that
the lower order ideal generated by $x$ in $R_{n}$
is isomorphic to
$R_b \times Q_1^{a_1} \times \cdots \times Q_n^{a_n}$.

\item[(D4)]
The poset $R_n$ has $N(n)$ minimal elements.
The sequence $(N(0), N(1), \ldots)$ is called the
{\em denominator sequence}.
\end{itemize}
\end{definition}
Observe that $R_{0}$ is the one element poset and thus $N(0) = 1$.
Also note if $x$ has type
$(b; a_1, \ldots, a_n)$
then
$a_{n-b+1} = \cdots = a_{n} = 0$.

Condition {\it (D3)} has a different formulation than
condition {\it (E3)}. The reason is that there could be
cases where the lower order ideal generated by an element
does not factor uniquely into the form
$R_b \times Q_1^{a_1} \times \cdots \times Q_n^{a_n}$.
However,
in the examples we consider the type of an element will be clear.

\begin{proposition}
Let 
${\bf R} = (R_{0}, R_{1}, \ldots)$ 
be an exponential Dowling structure
with associated exponential structure 
${\bf Q} = (Q_1, Q_2, \ldots)$.
The number of elements in $R_{n}$ of
type $(b;a_1, \ldots, a_n)$
is given by
\begin{equation}
\label{equation_pairs}
     \frac{N(n) \cdot s^n \cdot n!}
    {N(b) \cdot s^b \cdot b! \cdot
     (M(1) \cdot s \cdot 1!)^{a_1} \cdot a_1! 
     \cdots
     (M(n) \cdot s \cdot n!)^{a_n} \cdot a_n!
     }
\end{equation}
\end{proposition}
\begin{proof}
Consider pairs of elements $(x,y)$
satisfying
$y \leq x$, where
the element $x$ has type $(b;a_1, \ldots, a_n)$
and $y$ is a minimal element of~$R_n$.
We count such pairs in two ways.
The number of minimal elements $y \in R_n$ is given
by~$N(n)$.  Given such a minimal element $y$, the number
of $x$'s is given in Lemma~\ref{lemma_count_type}.
Alternatively,
we wish to count the number of $x$'s.
The number of $y$'s given an element $x$ equals the number
of minimal elements occurring in the lower order ideal generated by
$x$.
This equals the number of minimal elements in
$R_b \times Q_1^{a_1} \times \cdots \times Q_n^{a_n}$,
that is,
$N(b) \cdot M(1)^{a_1} \cdots M(n)^{a_n}$.
Thus the answer is as in~(\ref{equation_pairs}).
\end{proof}

Let ${\bf Q}$ be an exponential structure
and $r$ a positive integer.
Stanley defines the exponential structure ${\bf Q}^{(r)}$
by letting $Q^{(r)}_{n}$ be the subposet $Q_{r n}$ of all
elements $x$ of type $(a_{1}, a_{2}, \ldots)$ where
$a_{i} = 0$ unless $r$ divides $i$.
The denominator sequence of ${\bf Q}^{(r)}$ is
given by
$$   M^{(r)}(n) = \frac{M(r n) \cdot (r n)!}
                       {M(r)^{n} \cdot n! \cdot r!^{n}}  . $$

\begin{example}
{\rm
Let ${\bf R}$ be an exponential Dowling structure associated with the
exponential structure~${\bf Q}$.
Let $r$ be a positive integer and $k$ a non-negative integer.
Let $R^{(r,k)}_{n}$ be the subposet of $R_{r n + k}$ consisting of all
elements $x$ of type $(b; a_{1}, a_{2}, \ldots)$
such that
$b \geq k$, $b \equiv k \bmod r$
and
$a_{i} = 0$ unless $r$ divides $i$.
Then ${\bf R}^{(r,k)} = (R^{(r,k)}_{0}, R^{(r,k)}_{1}, \ldots)$
is an exponential Dowling structure associated with the
exponential structure ${\bf Q}^{(r)}$.
The minimal elements of $R^{(r,k)}_{n}$
are the elements of $R_{r n + k}$ having
types given by $b = k$, $a_{r} = n$
and $a_{i} = 0$ for $i \neq n$.
The denominator sequence of ${\bf R}^{(r,k)}$
is given by
$$   N^{(r,k)}(n) = 
     \frac{N(r n + k) \cdot (r n + k)! \cdot s^{(r-1) \cdot n}}
          {N(k) \cdot k!
            \cdot
           M(r)^{n} \cdot r!^{n} \cdot n! }    .   $$
}
\label{example_r_k}
\end{example}

Stanley~\cite{Stanley_e_s} proved the following structure theorem.

\begin{theorem}
(The Compositional Formula for Exponential Structures)
Let ${\bf Q} = (Q_{1}, Q_{1}, \ldots)$ be an exponential structure
with denominator sequence
$(M(1), M(2), \ldots)$.
Let
$f: \Ppp \rightarrow \Ccc$
and 
$g: \Nnn \rightarrow \Ccc$
be given functions such that $g(0) = 1$.
Define the function
$h : \Nnn \rightarrow \Ccc$ by
\begin{equation}
     h(n) = \sum_{x \in Q_{n}} f(1)^{a_1} \cdot f(2)^{a_2} \cdots
            f(n)^{a_n} \cdot g(a_1 + \cdots + a_n),
\label{equation_f_g_h}
\end{equation}
for $n \geq 1$, where
$\type(x) = (a_1, \ldots, a_n)$,
and $h(0) = 1$.
Define the formal power series $F, G, K \in \Ccc[[x]]$ by
\begin{eqnarray*}
F(x) & = & \sum_{n \geq 1} f(n) \cdot \frac{x^n}{M(n) \cdot n!} \\
G(x) & = & \sum_{n \geq 0} g(n) \cdot \frac{x^n}{n!} \\
H(x) & = & \sum_{n \geq 0} h(n) \cdot \frac{x^n}{M(n) \cdot n!} .
\end{eqnarray*}
Then
$H(x) = G(F(x))$.
\label{theorem_Stanley_compositional_formula}
\end{theorem}

For Dowling structures we have an analogous theorem.

\begin{theorem}
(The Compositional Formula for Exponential Dowling Structures)
Let ${\bf R} = (R_0, R_1, \ldots)$ be an exponential Dowling structure
with denominator sequence
$(N(0), N(1), \ldots)$ and
associated exponential structure
${\bf Q} = (Q_1, Q_2, \ldots)$ with
denominator sequence
$(M(1), M(2), \ldots)$.
Let
$f: \Ppp \rightarrow \Ccc$,
$g: \Nnn \rightarrow \Ccc$
and
$k: \Nnn \rightarrow \Ccc$ 
be given functions.
Define the function
$h : \Nnn \rightarrow \Ccc$ by
\begin{equation}
     h(n) = \sum_{x \in R_n} k(b) \cdot f(1)^{a_1} \cdot f(2)^{a_2} \cdots
            f(n)^{a_n} \cdot g(a_1 + \cdots + a_n),
\label{equation_f_g_h_k}
\end{equation}
for $n \geq 0$, where
$\type(x) = (b;a_1, \ldots, a_n)$.
Define the formal power series $F, G, K, H \in \Ccc[[x]]$ by
\begin{eqnarray*}
F(x) & = & \sum_{n \geq 1} f(n) \cdot \frac{x^n}{M(n) \cdot n!}\\
G(x) & = & \sum_{n \geq 0} g(n) \cdot \frac{x^n}{n!} \\
K(x) & = & \sum_{n \geq 0} k(n) \cdot \frac{x^n}{N(n) \cdot n!} \\
H(x) & = & \sum_{n \geq 0} h(n) \cdot \frac{x^n}{N(n) \cdot n!} 
\end{eqnarray*}
Then
$H(x) = K(x) \cdot G(1/s \cdot F(s \cdot x))$.
\label{theorem_compositional_formula}
\end{theorem}
\begin{proof_special}
By applying the compositional formula of generating functions
to the (exponential) generating functions
$1/s \cdot F(s x) = \sum_{n\geq 1} f(n)/(M(n) \cdot s) \cdot (s x)^{n}/n!$
and $G(x)$, 
we obtain
\begin{eqnarray*}
G(1/s \cdot F(s x))
  & = &
     \sum_{n \geq 0}
       \sum_{\pi \in \Pi_{n}}
         \prod_{B \in \pi}
             \frac{f(|B|)}{M(|B|) \cdot s}
           \cdot
             g(|\pi|)
           \cdot
	     \frac{(s x)^{n}}{n!}\\
  & = &
     \sum_{n \geq 0}
       \sum_{1 \cdot a_1 + \cdots + n \cdot a_n = n}
     \frac{s^{n} \cdot n!}
     {(M(1) \cdot s \cdot 1!)^{a_1} \cdot a_1!
            \cdots
      (M(n) \cdot s \cdot n!)^{a_n} \cdot a_n!}  
      \\
  &   &         \hspace{50 mm}
     \cdot
      f(1)^{a_1}
            \cdots
      f(n)^{a_n}
            \cdot
      g(a_1 + \cdots + a_n)
            \cdot
      \frac{x^{n}}{n!}  .
\end{eqnarray*}
Multiply this identity with
the (exponential) generating function
$K(x) = \sum_{n \geq 0} k(n)/N(n) \cdot x^{n}/n!$
to obtain
\begin{eqnarray*}
  &   &
K(x) \cdot G(1/s \cdot F(s x)) \\
  & = &
    \sum_{n \geq 0}
      \sum_{b = 0}^{n}
        \sum_{1 \cdot a_1 + \cdots + (n-b) \cdot a_{n-b} = n-b}
            {n \choose b}
          \cdot            
            \frac{k(b)}{N(b)} 
          \\
  &   &          \hspace{35 mm}
       \cdot
     \frac{s^{n-b} \cdot (n-b)!}
     {(M(1) \cdot s \cdot 1!)^{a_1} \cdot a_1!
            \cdots
      (M(n-b) \cdot s \cdot (n-b)!)^{a_{n-b}} \cdot a_{n-b}!}  
             \\
  &   &          \hspace{35 mm}
            \cdot
      f(1)^{a_1}
            \cdots
      f(n)^{a_n}
            \cdot
      g(a_{1} + \cdots + a_{n-b})
            \cdot
      \frac{x^{n}}{n!}    \\
  & = &
    \sum_{n \geq 0}
      \sum_{b = 0}^{n}
        \sum_{1 \cdot a_1 + \cdots + n \cdot a_n = n-b}
     \frac{s^{n} \cdot n!}
     { N(b) \cdot s^{b} \cdot b!
            \cdot
      (M(1) \cdot s \cdot 1!)^{a_1} \cdot a_1!
            \cdots
      (M(n) \cdot s \cdot n!)^{a_n} \cdot a_n!}  
             \\
  &   &         \hspace{50 mm}
            \cdot
      k(b)
            \cdot
      f(1)^{a_1}
            \cdots
      f(n)^{a_n}
            \cdot
      g(a_1 + \cdots + a_n)
            \cdot
      \frac{x^{n}}{n!}    \\
\hspace{10 mm}
  & = &
    \sum_{n \geq 0}
      \sum_{x \in R_{n}}
            k(b)
         \cdot
            {f(1)}^{a_1}
         \cdots
            {f(n)}^{a_n}
         \cdot
            g(a_1 + \cdots + a_n)
         \cdot
            \frac{x^{n}}{N(n) \cdot n!}    
    = H(x)  .
\hspace{30 mm} \qed
\end{eqnarray*}
\end{proof_special}

\begin{example}
{\rm
Let ${\bf R} = (R_0, R_1, \ldots)$ be an exponential Dowling structure
with denominator sequence
$(N(0), N(1), \ldots)$ and
associated exponential structure
${\bf Q} = (Q_1, Q_2, \ldots)$ with
denominator sequence
$(M(1), M(2), \ldots)$.
Let $V_{n}(t)$ be the polynomial
$$   V_{n}(t)
   =
     \sum_{x \in Q_{n}} t^{\rho(x,\ho)}  .  $$
In Example~5.5.6 in~\cite{Stanley_EC_II}
Stanley obtains the generating function
$$   \sum_{n \geq 0} V_{n}(t) \cdot \frac{x^n}{M(n) \cdot n!} 
   =
     \exp\left(
            \sum_{n \geq 1} \frac{x^n}{M(n) \cdot n!} 
         \right)^{t}   ,  $$
by setting $f(n) = 1$ and $g(n) = t^{n}$
in Theorem~\ref{theorem_Stanley_compositional_formula}.
Similarly, defining $W_{n}(t)$ by
$$   W_{n}(t)
   =
     \sum_{x \in R_{n}} t^{\rho(x,\ho)}  ,  $$
we obtain
$$   \sum_{n \geq 0} W_{n}(t) \cdot \frac{x^n}{N(n) \cdot n!} 
   =
         \left(
           \sum_{n \geq 0} \frac{x^n}{N(n) \cdot n!}
         \right)
      \cdot
         {\exp\left(
                  \sum_{n \geq 1} \frac{(s \cdot x)^n}{M(n) \cdot n!} 
              \right)}^{\frac{t}{s}}   ,  $$
by setting $f(n) = 1$, 
$g(n) = t^{n}$ and $k(n) = 1$
in Theorem~\ref{theorem_compositional_formula}.
}
\end{example}

\begin{corollary}
Let ${\bf R} = (R_0, R_1, \ldots)$ be an exponential Dowling structure
with denominator sequence
$(N(0), N(1), \ldots)$ and
associated exponential structure
${\bf Q} = (Q_1, Q_2, \ldots)$ with
denominator sequence
$(M(1), M(2), \ldots)$.
Then the M\"obius function of
the posets
$Q_{n} \cup \{\hz\}$,
respectively
$R_{n} \cup \{\hz\}$,
has the generating function:
\begin{eqnarray}
\sum_{n \geq 1}
    \mu(Q_{n} \cup \{\hz\})
  \cdot
    \frac{x^{n}}{M(n) \cdot n!}
  & = &
    -
\ln\left(
\sum_{n \geq 0}
    \frac{x^{n}}{M(n) \cdot n!}
\right)    , 
\label{equation_one_Mobius} \\
\sum_{n \geq 0}
    \mu(R_{n} \cup \{\hz\})
  \cdot
    \frac{x^{n}}{N(n) \cdot n!}
  & = &
    -
\left(
\sum_{n \geq 0}
    \frac{x^{n}}{N(n) \cdot n!}
\right)
  \cdot
\left(
\sum_{n \geq 0}
    \frac{(s \cdot x)^{n}}{M(n) \cdot n!}
\right)^{-1/s} . 
\label{equation_two_Mobius}
\end{eqnarray}
\label{corollary_Mobius}
\end{corollary}
\begin{proof}
Setting $f(n) = 1$ and $g(n) = (-1)^{n-1} \cdot (n-1)!$ and using
that
$$
\mu(Q_{n} \cup \{\hz\})
    =
  -
\sum_{x \in Q_{n}}
    \mu(x,\ho)  
    =
  -
    \sum_{x \in Q_{n}} 
       g(a_1 + \cdots + a_n) ,
$$
equation~(\ref{equation_one_Mobius}) follows
by Theorem~\ref{theorem_Stanley_compositional_formula}.
Similarly, to prove the second
identity~(\ref{equation_two_Mobius}),
redefine $g(n)$ to
be the M\"obius function of the Dowling lattice
$L_{n}$ of rank $n$, that is,
$$  g(n) = (-1)^{n} \cdot 1 \cdot (s+1) \cdot (2 \cdot s + 1)
                            \cdots ((n-1) \cdot s + 1) .  $$
By the binomial theorem we have
$\sum_{n \geq 0} g(n) \frac{x^{n}}{n!} 
  =
  (1 + s \cdot x)^{-1/s}$.
Moreover, let $k(n) = 1$.
Using the recurrence
$$
\mu(R_{n} \cup \{\hz\})
    = 
  -
\sum_{x \in R_{n}}
    \mu(x,\ho)        
    =
\sum_{x \in R_{n}} 
    g(a_1 + \cdots + a_n) ,
$$
and Theorem~\ref{theorem_compositional_formula},
the result follows.
\end{proof}

\section{The M\"obius function of restricted structures}
\label{section_Mobius_function}
\setcounter{equation}{0}

Let $I$ be a subset of the positive integers $\Ppp$.
For an exponential structure
${\bf Q} = (Q_{1}, Q_{2}, \ldots)$
define the {\em restricted poset} $Q_{n}^{I}$ to be all elements $x$
in $Q_{n}$ whose type
$(a_{1}, \ldots, a_{n})$
satisfies $a_{i} > 0$ implies $i \in I$.
For $n \in I$ let $\mu_{I}(n)$ denote the
M\"obius function of the poset $Q_{n}^{I}$ with a $\hz$ adjoined,
that is, the poset
$Q_{n}^{I} \cup \{\hz\}$.
For $n \not\in I$ let $\mu_{I}(n) = 0$.

For any positive integer $n$ define
$$ m_{n} = \sum_{x \in Q_{n}^{I} \cup \{\hz\}} \mu_{I}(\hz,x) 
         =  1 + \sum_{x \in Q_{n}^{I}} \mu_{I}(\hz,x) .  $$
Observe that for $n \in I$ we have 
that $Q_{n}^{I}$ has a maximal element and
hence $m_{n} = 0$.
Especially for $n \not\in I$ we have
the expansion
$$ m_{n} =   1
-
 \sum_{\sum_{i \in I} i \cdot a_{i} = n}
  (-1)^{\sum_{i \in I} a_{i}}
     \frac{M(n) \cdot n!}
    {(M(1) \cdot 1!)^{a_1} \cdot a_1! 
     \cdots
     (M(n) \cdot n!)^{a_n} \cdot a_n!}
  \cdot
  \mu_{I}(1)^{a_1} \cdots  \mu_{I}(n)^{a_n}   .  $$

The following theorem
was inspired by
work of Linusson~\cite{Linusson}.
\begin{theorem}
$$
  \sum_{i \in I} \mu_{I}(i) \frac{x^{i}}{M(i) \cdot i!}
   = 
 -
  \ln \left(
      \sum_{n \geq 0}    \frac{x^{n}}{M(n) \cdot n!} 
   -
      \sum_{n \not\in I}   m_{n} \cdot  \frac{x^{n}}{M(n) \cdot n!} 
      \right)        .
$$
\label{theorem_I}
\end{theorem}

\begin{proof}
Expand the product
\begin{eqnarray*}
 - 1 +  \prod_{i \in I} \exp\left(-\mu_{I}(i) \frac{x^{i}}{M(i) \cdot i!}
                           \right)
  & = &
 - 1 +  \prod_{i \in I} 
\left(1 - \mu_{I}(i) \frac{x^{i}}{M(i) \cdot i!}
               + \frac{1}{2} \cdot
                 \left(\mu_{I}(i) \frac{x^{i}}{M(i) \cdot i!}\right)^{2}
               - \cdots    \right)  \\
  & = &
 \sum_{n \geq 1}
 \sum_{\sum_{i \in I} i \cdot a_{i} = n}
   \prod_{i \in I}
   \frac{1}{a_{i}!} \cdot
   \left(- \mu_{I}(i) \frac{x^{i}}{M(i) \cdot i!} \right)^{a_i} \\
  & = &
 \sum_{n \geq 1}
 \sum_{\sum_{i \in I} i \cdot a_{i} = n}
  (-1)^{\sum_{i \in I} a_{i}}
  \cdot
\left(  \mu_{I}(1)^{a_1} \cdots  \mu_{I}(n)^{a_n}  \right) \\
  & \cdot &
   \frac{M(n) \cdot n!}
        {(M(1) \cdot 1!)^{a_1} \cdot a_{1}! \cdots 
         (M(n) \cdot n!)^{a_n} \cdot a_{n}!}
            \frac{x^{n}}{M(n) \cdot n!} \\
  & = &
      \sum_{n \geq 1}   (1 - m_{n}) \cdot  \frac{x^{n}}{M(n) \cdot n!} \\
  & = &
      \sum_{n \geq 1}    \frac{x^{n}}{M(n) \cdot n!} 
   -
      \sum_{n \not\in I}   m_{n} \cdot  \frac{x^{n}}{M(n) \cdot n!} .
\end{eqnarray*}
The result now follows.
\end{proof}

Let $I$ be a subset of the positive integers $\Ppp$
and $J$ be a subset of the natural numbers $\Nnn$.
For an exponential Dowling structure
${\bf R} = (R_{0}, R_{1}, \ldots)$,
define the {\em restricted poset} $R_{n}^{I,J}$ to be all elements $x$
in $R_{n}$ whose type
$(b; a_{1}, \ldots, a_{n})$
satisfies
$b \in J$
and
$a_{i} > 0$ implies $i \in I$.

For $n \in J$ define $\mu_{I,J}(n)$ to be the
M\"obius function of the poset 
$R_{n}^{I,J} \cup \{\hz\}$, that is, 
$R_{n}^{I,J}$ with a minimal element $\hz$ adjoined.
For $n \not\in J$ let $\mu_{I,J}(n) = 0$.
Define for any non-negative integer $n$ 
$$ p_{n} = \sum_{x \in R_{n}^{I,J} \cup \{\hz\}} \mu_{I,J}(\hz,x) 
         = 1 + \sum_{x \in R_{n}^{I,J}} \mu_{I,J}(\hz,x)  . $$
Observe that for $n \in J$ we have $p_{n} = 0$
since the poset $R_{n}^{I,J} \cup \{\hz\}$ has a maximal element.
For $n \not\in J$ we have
\begin{eqnarray*}
 p_{n} 
  & = &
 1
+
 \sum_{(b;a_{1}, \ldots, a_{n})}
(-1)^{a_1 + \cdots + a_{n}}
   \cdot
     \frac{N(n) \cdot s^{n} \cdot n!}
          {N(b) \cdot s^{b} \cdot b!
               \cdot
     (M(1) \cdot s \cdot 1!)^{a_1} \cdot a_1! 
     \cdots
     (M(n) \cdot s \cdot n!)^{a_n} \cdot a_n!} \\
  &   &
\hspace*{50 mm}
  \cdot
  \mu_{I,J}(b)
    \cdot
  \mu_{I}(1)^{a_1} \cdots  \mu_{I}(n)^{a_n}   ,  
\end{eqnarray*}
where the sum is over all types
$(b; a_{1}, \ldots, a_{n})$ where
$b \in J$,
$a_{i} > 0$ implies $i \in I$,
and
$b + \sum_{i \in I} i \cdot a_{i} = n$.

\begin{theorem}
\begin{eqnarray*}
   \sum_{b \in J} \mu_{I,J}(b) \cdot
                      \frac{x^{b}}{N(b) \cdot b!}  
  & = &
\frac{\displaystyle
  -
  \sum_{n \geq 0}
         \frac{x^{n}}{N(n) \cdot n!}
  +
  \sum_{n \not\in J}
       p_{n} \cdot 
         \frac{x^{n}}{N(n) \cdot n!}
}
{\displaystyle
\left(
      \sum_{n \geq 0}    \frac{(s \cdot x)^{n}}{M(n) \cdot n!} 
   -
      \sum_{n \not\in I}   m_{n} \cdot  \frac{(s \cdot x)^{n}}{M(n) \cdot n!} 
 \right)^{1/s}}
\end{eqnarray*}
\label{theorem_I_J}
\end{theorem}
\begin{proof}
By similar reasoning as in the proof
of Theorem~\ref{theorem_I}, we have
$$
   \exp\left( -
       \sum_{i \in I} \mu_{I}(i) \cdot
                      \frac{x^{i}}{M(i) \cdot s \cdot i!}
       \right)
  =
   \sum_{(a_{1}, \ldots, a_{n})}
        \prod_{i \in I}
             \frac{1}{a_{i}!} \cdot
               \left(
                     - \frac{\mu_{I}(i) \cdot x^{i}}{M(i) \cdot s \cdot i!}
               \right)^{a_{i}}      .
$$
Multiplying with
$\sum_{b \in J} \mu_{I,J}(b) \cdot
                      \frac{x^{b}}{N(b) \cdot s^{b} \cdot b!}$
and expanding, we obtain
\begin{eqnarray*}
  &    &
\left(\sum_{b \in J} \mu_{I,J}(b) \cdot
                      \frac{x^{b}}{N(b) \cdot s^{b} \cdot b!}
\right)
  \cdot
   \exp\left( -
       \sum_{i \in I} \mu_{I}(i) \cdot
                      \frac{x^{i}}{M(i) \cdot s \cdot i!}
       \right)    \\
  & = &
  \sum_{n \geq 0}
    \sum_{b \in J}
    \sum_{(a_1, \ldots, a_{n-b})}
    \left( \frac{ \mu_{I,J}(b) \cdot x^{b}}{N(b) \cdot s^{b} \cdot b!}
    \right)
   \cdot
        \prod_{i \in I}
             \frac{1}{a_{i}!} \cdot
               \left(
                     - \frac{\mu_{I}(i) \cdot x^{i}}{M(i) \cdot s \cdot i!}
               \right)^{a_{i}}    \\
  & = &
  \sum_{n \geq 0}
       (p_{n} - 1) \cdot 
         \frac{x^{n}}{N(n) \cdot s^{n} \cdot n!}
\end{eqnarray*}
Substituting $x \longmapsto s x$, we can rewrite this
equation as
\begin{eqnarray*}
  &    &
   \sum_{b \in J} \mu_{I,J}(b) \cdot
                      \frac{x^{b}}{N(b) \cdot b!}  \\
  & = &
\left(
  \sum_{n \geq 0}
       (p_{n} - 1) \cdot 
         \frac{x^{n}}{N(n) \cdot n!}
\right)
  \cdot
   \exp\left( \frac{1}{s} \cdot
       \sum_{i \in I} \mu_{I}(i) \cdot
                      \frac{(s \cdot x)^{i}}{M(i) \cdot i!}
       \right)    \\
  & = &
\left(
  \sum_{n \geq 0}
       (p_{n} - 1) \cdot 
         \frac{x^{n}}{N(n) \cdot n!}
\right)
  \cdot
   \exp\left(
       \sum_{i \in I} \mu_{I}(i) \cdot
                      \frac{(s \cdot x)^{i}}{M(i) \cdot i!}
       \right)^{1/s}   
\end{eqnarray*}
By applying
Theorem~\ref{theorem_I}
to the last term, the result follows.
\end{proof}

As a corollary to
Theorems~\ref{theorem_I}
and~\ref{theorem_I_J},
we have
\begin{corollary}
Let $I \subseteq \Ppp$ be a semigroup
and $J \subseteq \Nnn$ such that
$I + J \subseteq J$.
Then the M\"obius function of
the restricted poset
$Q_{n}^{I} \cup \{\hz\}$
and
$R_{n}^{I,J} \cup \{\hz\}$
respectively
has the generating function:
\begin{eqnarray}
\sum_{n \in I}
    \mu_{I}(n)
  \cdot
    \frac{x^{n}}{M(n) \cdot n!}
  & = &
    -
\ln\left(
\sum_{n \in I \cup \{0\}}
    \frac{x^{n}}{M(n) \cdot n!}
\right)    , 
\label{equation_one_corollary_semigroup} \\
\sum_{n \in J}
    \mu_{I,J}(n)
  \cdot
    \frac{x^{n}}{N(n) \cdot n!}
  & = &
    -
\left(
\sum_{n \in J}
    \frac{x^{n}}{N(n) \cdot n!}
\right)
  \cdot
\left(
\sum_{n \in I \cup \{0\}}
    \frac{(s \cdot x)^{n}}{M(n) \cdot n!}
\right)^{-1/s} . 
\label{equation_two_corollary_semigroup} 
\end{eqnarray}
\label{corollary_semigroup}
\end{corollary}
\begin{proof}
The semigroup condition implies that
the poset $Q_{n}^{I}$ is empty when $n \not\in I$
and hence $m_{n} = 1$. Similarly,
the other condition implies that
the poset $R_{n}^{I}$ is empty when $n \not\in J$,
so $p_{n} = 1$.
\end{proof}

Let ${\bf D}$ be the Dowling structure consisting of 
the Dowling lattices, that is,
${\bf D} = (L_0, L_1, \ldots)$.

\begin{proposition}
For the exponential Dowling structure
${\bf D}^{(r,k)}$ we have
$$  \sum_{n \geq 0}
            \mu\left(D_{n}^{(r,k)} \cup \{\hz\}\right)
         \cdot
            \frac{x^{r n + k}}{(r n + k)!}  
  =
  \left(
    \sum_{n \geq 0} \frac{x^{r n + k}}{(r n + k)!}
  \right)
\cdot
  \left(
    \sum_{n \geq 0} \frac{(s \cdot x)^{r n}}{(r n)!}
  \right)^{-1/s}                                    .  $$
\label{proposition_M}
\end{proposition}

This can be proven from 
Corollary~\ref{corollary_semigroup}
using $I = r \cdot \Ppp$ and $J = k + r \cdot \Nnn$.
This also follows from
Corollary~\ref{corollary_Mobius}
by using the Dowling structure ${\bf D}^{(r,k)}$.

When $r=1$ we have the following corollary.
\begin{corollary}
Let $k \geq 1$.
Then the M\"obius function
of the poset $D_{n}^{(1,k)} \cup \{\hz\}$
is given by
$$   \mu\left(D_{n}^{(1,k)} \cup \{\hz\}\right)
   =
     (-1)^{n} \cdot {{n+k-1} \choose {k-1}} . $$
Furthermore, the M\"obius function does not depend on the order $s$.
\end{corollary}
\begin{proof_special}
We have
$$  \sum_{n \geq 0}
            \mu\left(D_{n}^{(1,k)} \cup \{\hz\}\right)
         \cdot
            \frac{x^{n+k}}{(n+k)!}  
  =
  \left(
    \sum_{n \geq 0} \frac{x^{n+k}}{(n+k)!}
  \right)
\cdot
  \exp(-x)         . $$
Differentiate with respect to $x$ gives
\begin{eqnarray*}
    \sum_{n \geq 0}
            \mu\left(D_{n}^{(1,k)} \cup \{\hz\}\right)
         \cdot
            \frac{x^{n+k-1}}{(n+k-1)!}  
  & = &
  \left(
    \sum_{n \geq 0} \frac{x^{n+k-1}}{(n+k-1)!}
  \right)
\cdot
  \exp(-x)        
-
  \left(
    \sum_{n \geq 0} \frac{x^{n+k}}{(n+k)!}
  \right)
\cdot
  \exp(-x)        \\
  & = &
    \frac{x^{k-1}}{(k-1)!}
\cdot
  \exp(-x)         \\
  & = &
    \sum_{n \geq 0} (-1)^{n} \cdot
           {{n+k-1} \choose {k-1}}
             \cdot \frac{x^{n+k-1}}{(n+k-1)!}  .
\hspace{30 mm} \qed
\end{eqnarray*}
\end{proof_special}

When $r=2$ we can express the generating function
for the M\"obius function
in terms of hyperbolic functions. We have two cases,
depending on whether $k$ is even or odd.
\begin{corollary}
The M\"obius function
of the poset $D_{n}^{(2,k)} \cup \{\hz\}$
is given by
\begin{eqnarray}
    \sum_{n \geq 0}
            \mu\left(D_{n}^{(2,2j)} \cup \{\hz\}\right)
         \cdot
            \frac{x^{2n+2j}}{(2n+2j)!}  
  & = &
  \left(
    \cosh(x)
   - 
    \sum_{i = 0}^{j-1} \frac{x^{2i}}{(2i)!}
  \right)
\cdot
\sech\left( s \cdot x  \right)^{1/s}  
                                              , \\
    \sum_{n \geq 0}
            \mu\left(D_{n}^{(2,2j+1)} \cup \{\hz\}\right)
         \cdot
            \frac{x^{2n+2j+1}}{(2n+2j+1)!}  
  & = &
  \left(
    \sinh(x)
   - 
    \sum_{i = 0}^{j-1} \frac{x^{2i+1}}{(2i+1)!}
  \right)
\cdot
\sech\left( s \cdot x  \right)^{1/s}.
\end{eqnarray}
\label{corollary_hyperbolic}
\end{corollary}

\section{Permutations and partitions with restricted block sizes}
\label{section_permutations}
\setcounter{equation}{0}

For a permutation $\sigma = \sigma_{1}\sigma_{2}\cdots\sigma_{n}$
in the symmetric group
$\Ssss_{n}$ define the descent set of $\sigma$ to be
the set
$\{ i \: : \: \sigma_{i} < \sigma_{i+1}\}$.
An equivalent notion is 
the {\em descent word} of $\sigma$,
which is the $\ab$-word
$u = u_{1} u_{2} \cdots u_{n-1}$ of degree $n-1$ where
$u_{i} = \av$ if $\sigma_{i} < \sigma_{i+1}$
and
$u_{i} = \bv$ otherwise.
For an $\ab$-word $u$ of length $n-1$
let $\Des{u}$ be the number of permutations $\sigma$
in $\Ssss_{n}$ with descent word $u$.
Similarly, define the $q$-analogue $\Desq{u}$
to be the sum
$$  \Desq{u}
  =
    \sum_{\sigma} q^{\mbox{\scriptsize\rm inv}(\sigma)}   , $$
where the sum ranges over all permutations $\sigma$
in $\Ssss_{n}$ with descent word $u$
and $\mbox{\rm inv}(\sigma)$ is the number of
inversions of $\sigma$.
Let $[n]$ denote $1 + q + \cdots + q^{n-1}$
and $[n]! = [1] \cdot [2] \cdots [n]$.
Finally, let $\qb{n}{k}$ denote the Gaussian coefficient
$[n]!/([k]! \cdot [n-k]!)$.

\begin{lemma}
For two $\ab$-words $u$ and $v$ of degree $n-1$, respectively $m-1$,
the following identity holds:
$$   \qb{n+m}{n}
   \cdot
     \Desq{u} \cdot \Desq{v}
 = 
     \Desq{u \cdot \av \cdot v}
   +
     \Desq{u \cdot \bv \cdot v} .  $$
\end{lemma}

This is  ``the Multiplication Theorem'' due to
MacMahon~\cite[Article~159]{MacMahon}.
Using this identity, we obtain 
the following lemma for Eulerian generating functions.
\begin{lemma}
Let $\left( u_{n} \right)_{n \geq 1}$
and $\left( v_{n} \right)_{n \geq 1}$
be two sequences of $\ab$-words such that
the $n$th word has degree $n-1$.
Then the following Eulerian generating function identity holds:
\begin{eqnarray*}
  &   &
\left(
\sum_{n \geq 1} c_{n} \cdot \Desq{u_{n}} \cdot \frac{x^{n}}{[n]!}
\right)
\cdot
\left(
\sum_{n \geq 1} d_{n} \cdot \Desq{v_{n}} \cdot \frac{x^{n}}{[n]!}
\right) \\
  & = &
\sum_{n \geq 2} \sum_{{i+j=n} \atop {i,j \geq 1}}
       c_{i} \cdot d_{j} \cdot
      \left(\Desq{u_{i} \cdot \av \cdot v_{j}}
          +
            \Desq{u_{i} \cdot \bv \cdot v_{j}}\right)
 \cdot \frac{x^{n}}{[n]!} .
\end{eqnarray*}
\end{lemma}

Now we obtain the following proposition.
In the special case when $w=\av^{i}$, where
$0 \leq i \leq r-1$, the result is due to 
Stanley~\cite{Stanley_binomial}. See
also~\cite[Section~3.16]{Stanley_EC_I}.

\begin{proposition}
Let $w$ be an $\ab$-word of degree $k-1$.
Then Eulerian generating function for
the descent statistic $\Desq{(\av^{r-1} \bv)^{n} \cdot w}$
is given by
$$  \sum_{n \geq 0}
            (-1)^{n} \cdot
            \Desq{(\av^{r-1} \bv)^{n} \cdot w} \cdot
            \frac{x^{r n + k}}{[r n + k]!}  
  =
  \frac{{\displaystyle
    \sum_{n \geq 0} \Desq{\av^{r n} \cdot w} \cdot
            \frac{x^{r n + k}}{[r n + k]!}  
          }}{{\displaystyle
    \sum_{n \geq 0} \frac{x^{r n}}{[r n]!}
              }}                  .  $$
\label{proposition_permutations}
\end{proposition}
\begin{proof}
Consider the following product of generating functions:
\begin{eqnarray*}
  &   &
\left(
    \sum_{n \geq 1} 
            \Desq{\av^{r n-1}} \cdot
            \frac{x^{r n}}{[r n]!}
\right)
  \cdot
\left(
    \sum_{n \geq 0} 
            (-1)^{n} \cdot
            \Desq{(\av^{r-1} \bv)^{n} \cdot w} \cdot
            \frac{x^{r n + k}}{[r n + k]!}  
\right) \\
       & = &
\sum_{n \geq 0} 
\left(
\sum_{{i+j = n} \atop {i \geq 1}}
(-1)^{j} \cdot
\left(
      \Desq{\av^{r i} (\av^{r-1} \bv)^{j} \cdot w}
    +
      \Desq{\av^{r (i-1)} (\av^{r-1} \bv)^{j+1} \cdot w}
\right)
\right)
\cdot 
            \frac{x^{r n + k}}{[r n + k]!}   \\
       & = &
\sum_{n \geq 0} 
\left(
      \Desq{\av^{r n} \cdot w}
    +
      (-1)^{n-1} \cdot
      \Desq{(\av^{r-1} \bv)^{n} \cdot w}
\right)
\cdot 
            \frac{x^{r n + k}}{[r n + k]!}   .
\end{eqnarray*}
Now add 
$\sum_{n \geq 0}
            (-1)^{n} \cdot
            \Desq{(\av^{r-1} \bv)^{n} \cdot w} \cdot
            x^{r n + k}/[r n + k]!$
to both sides and the desired identity is established.
\end{proof}

For $r$ a positive integer and $n$ a non-negative 
integer let $m = r n$. Define
the poset $\Pi_{m}^{r}$ to be
the collection of all partitions $\pi$
of the set
$\{1, \ldots, m\}$
such that each block size is divisible by $r$
together with a minimal element $\hz$ adjoined.
This is the well-known and well-studied
{\em $r$-divisible partition lattice}.
See~\cite{Calderbank_Hanlon_Robinson, Sagan, Stanley_e_s, Wachs_1}.
Other restrictions of the partition lattice and the Dowling lattice
can be found in~\cite{Bjorner_Sagan, Gottlieb_I, Gottlieb_II}.

A natural extension of the $r$-divisible partition lattice
is the following.
For $r$ a positive integer, and $n$ and $j$ non-negative integers,
let $m = r n + j$.
Define the poset $\Pi_{m}^{r,j}$ to be
the collection of all partitions $\pi$
of the set
$\{1, \ldots, m\}$
such that
\vspace*{-2 mm}
\begin{itemize}
\item[(i)] a block $B$ of $\pi$ containing the element $m$
must have cardinality at least $j$,
\vspace*{-2 mm}
\item[(ii)] a block $B$ of $\pi$ not containing the element $m$
must have cardinality divisible by $r$,
\end{itemize}
\vspace*{-2 mm}
together with a minimal element $\hz$ adjoined to the poset.
We order all such partitions in the usual way by refinement.
For instance, $\Pi_{m}^{1,1}$ is the classical
partition lattice $\Pi_{m}$ with $\hz$ adjoined.
Observe that the poset $\Pi_{m}^{r,j} -\{\hz\}$
is a filter (upper order ideal) of the partition lattice~$\Pi_{m}$.
Hence~$\Pi_{m}^{r,j}$ is a finite semi-join lattice
and we can conclude that it is a lattice.
The same argument holds for~$\Pi_{m}^{r}$.

By combining Propositions~\ref{proposition_M}
and~\ref{proposition_permutations},
we obtain the next result.
\begin{theorem}
Let $r$ and $k$ be positive integers and $n$ a non-negative 
integer and let $m = r n + k + 1$.
Then the M\"obius function of
the lattice $\Pi_{m}^{r,k+1}$
is given by
the sign $(-1)^{n}$ times
the number of permutations
on $m-1$ elements
with the descent set
$\{r, 2 r, \ldots, n r\}$,
that is,
$$      \mu(\Pi_{m}^{r,k+1})
    =
        (-1)^{n}
      \cdot
        \Des{(\av^{r-1} \bv)^{n} \cdot \av^{k-1}}  .  $$
\label{theorem_m_plus_1}
\end{theorem}
\begin{proof}
Begin to observe that $\Pi_{m}^{r,k+1}$ is
isomorphic to the poset $D^{(r,k)}_{n}$ when $s=1$.
Namely, remove the element $m$ from the block $B$ that contains
this element and rename this block to be the zero block.
The result follows now by observing
that setting $w = \av^{k-1}$ and $q=1$ in
Proposition~\ref{proposition_permutations}
gives the same generating function as
setting $s=1$ in
Proposition~\ref{proposition_M}.
\end{proof}

For completeness, we also consider the case $j=1$.
\begin{theorem}
Let $r$ and $n$ be positive integers and let $m = r n + 1$.
Then the M\"obius function of
the lattice $\Pi_{m}^{r,1}$
is $0$.
\label{theorem_j_1}
\end{theorem}
\begin{proof}
This follows directly from
Proposition~\ref{proposition_M}
by setting $k=0$ and $s=1$.
A direct combinatorial argument is the following.
Each of the atoms of the lattice $\Pi_{m}^{r,1}$
has the element $m$ in a singleton block.
The same holds for the join of all
the atoms and hence the join of all
the atoms is not the maximal element $\ho$ of the lattice.
Thus by Corollary~3.9.5 in~\cite{Stanley_EC_I} the result is obtained.
\end{proof}

Setting $k = r-1$ in Theorem~\ref{theorem_m_plus_1},
we obtain the following corollary due
to Stanley~\cite{Stanley_e_s}.
\begin{corollary}
For $r \geq 2$ and $m = r n$
the M\"obius function of
the $r$-divisible partition lattice $\Pi_{m}^{r}$
is given by
the sign $(-1)^{n-1}$
times the number of
permutations of $r n - 1$ elements
with the descent set
$\{r, 2 r, \ldots, (n-1) r\}$,
that is,
$$      \mu(\Pi_{m}^{r})
    =
        (-1)^{n-1}
      \cdot
        \Des{(\av^{r-1} \bv)^{n-1} \cdot \av^{r-2}}  .  $$
\end{corollary}
When $r=2$ this corollary reduces to
$(-1)^{n-1} \cdot E_{2n-1}$, where $E_{i}$ denotes
the $i$th Euler number.
This result is originally due to
G.\ S.\ Sylvester~\cite{Sylvester}.
The odd indexed Euler numbers are known as
the {\em tangent numbers}
and the even indexed ones as
the {\em secant numbers}.
Setting $r=2$ and $k=2$
in Theorem~\ref{theorem_m_plus_1}
we obtain 
that the M\"obius function of the
partitions where all blocks have even size
except the block containing the largest element,
which has an odd size greater than or equal to three,
is given by 
the secant numbers,
that is, $(-1)^{n-1} \cdot E_{2n}$.

\section{EL-labeling}
\label{section_EL}
\setcounter{equation}{0}

It is a natural question
to ask if the poset $\Pi_{m}^{r,j}$ occurring
in Theorems~\ref{theorem_m_plus_1}
and~\ref{theorem_j_1} is
$EL$-shellable. The answer is positive.
An $EL$-labeling that works is the one using Wachs'
$EL$-labeling~\cite{Wachs_1} for the $r$-divisible
partition lattice $\Pi_{m}^{r}$,
which we state here for the extended
partition lattice $\Pi_{m}^{r,j}$.
Let $r$ and $j$ be positive integers and $n$ a non-negative 
integer and let $m = r n + j$.
Define the labeling $\lambda$ as follows.
First consider the edges in the Hasse diagram not adjacent to
the minimal element $\hz$.
Let $x$ and $y$ be two elements in $\Pi_{m}^{r,k+1} - \{\hz\}$ such
that $x$ is covered by $y$ and $B_{1}$ and $B_{2}$ are the blocks of
$x$ that are merged to form the partition $y$.
Assume that $\max(B_{1}) < \max(B_{2})$.
Set
\begin{equation}
   \lambda(x,y)
  =
   \left\{
     \begin{array}{c l}
        -\max(B_{1}) & \mbox{ if } \max(B_{1}) > \min(B_{2}), \\
        \max(B_{2})  & \mbox{ otherwise.}
      \end{array} 
   \right.
\label{equation_lambda_1}
\end{equation}
Now consider the edges between the minimal element
$\hz$ and the atoms. There are
$M = (m-1)!/(n! \cdot r!^{n} \cdot (j-1)!)$
number of atoms.
For each atom $a = \{B_{1}, B_{2}, \ldots, B_{n+1}\}$
order the blocks such that $\min(B_{1}) < \min(B_{2}) < 
\cdots < \min(B_{n+1})$.
Let $\widetilde{a}$ be the permutation in
$\Ssss_{m}$ that is obtained by going through the blocks
in order and writing down the elements
in each block in increasing order.
For instance, for the atom $a = 16|23|459|78$ we obtain
the permutation $\widetilde{a} = 162345978$.
It is straightforward to see that different atoms give rise
to different permutations by considering where the largest
element $m$ is.
Finally, order the atoms $a_{1} < a_{2} < \cdots < a_{M}$
such that the permutations
$\widetilde{a_{1}} < \widetilde{a_{2}} < \cdots < \widetilde{a_{M}}$
are ordered in lexicographic order.
Define the label of the edge from the minimal element to an atom
by
\begin{equation}
    \lambda(\hz,a_{i})
  =
    0_{i} 
\label{equation_lambda_2}
\end{equation}
Order the labels by
$$  \{-m < -(m-1) < \cdots < -1
    < 0_{1} < 0_{2} < \cdots < 0_{M} <
    1 < \cdots < m \}  .  $$

Let $A_{m}^{r,j}$ be the collection
of all permutations $\sigma \in \Ssss_{m}$ such
that the descent set of $\sigma$ is
$\{r, 2r, \ldots, n r\}$ and $\sigma(m) = m$.
Note that when $j=1$ there are no such permutations
since the condition $\sigma(m) = m$ forces
$n r$ to be an ascent. Given
a permutation $\sigma \in A_{m}^{r,j}$,
let $t_{1}, \ldots, t_{n}$ be
the permutation of $1, \ldots, n$ such that
$$ \sigma(r t_{1}) > 
   \sigma(r t_{2}) > 
         \cdots    >
   \sigma(r t_{n})      .  $$
Define the maximal chain $f_{\sigma}$
in $\Pi_{m}^{r,j}$ whose $i$-block
partition is obtained by splitting $\sigma$
at $r t_{1}$, $r t_{2}$, $\ldots$, $r t_{i-1}$.
As an example, for $\sigma = 562418379$
where $r = 2$, $n=3$, $j=3$ and $m=9$, we have
the maximal chain
$$  f_{562418379}
  =
  \{\hz <
    56|24|18|379 <
    56|2418|379 <
    562418|379 <
    562418379 = \ho\} .  $$
Observe that different permutations
in $A_{m}^{r,j}$ give different maximal chains.

\begin{theorem}
The labeling 
$(\lambda(x,y), -\rho(x))$
where $\lambda$ is defined in equations~(\ref{equation_lambda_1})
and~(\ref{equation_lambda_2}),
$\rho$ denotes the rank function and
the ordering is lexicographic on the pairs,
is an $EL$-labeling for the poset
$\Pi_{m}^{r,j}$.
The falling maximal chains are given by
$\{ f_{\sigma} \: : \: \sigma \in A_{m}^{r,j}\}$.
\end{theorem}
The proof that this labeling is an $EL$-labeling mimics the
proof of Theorem~5.2 in Wachs' paper~\cite{Wachs_1} and hence is omitted.

We distinguish between the cases $j=1$ and $j \geq 2$
in the following two corollaries.
\begin{corollary}
The chain complex of $\Pi_{m}^{r,1}$
is contractible.
\end{corollary}

\begin{corollary}
The chain complex of $\Pi_{m}^{r,j}$
is homotopy equivalent to a wedge of 
$\Des{(\av^{r-1} \bv)^{n} \cdot \av^{j-2}}$ number of
$(n-1)$-dimensional spheres.
Hence all the poset homology of the
poset
$\Pi_{m}^{r,j}$
is concentrated in the top homology
which has rank
$\Des{(\av^{r-1} \bv)^{n} \cdot \av^{j-2}}$
\end{corollary}

\section{Concluding remarks}
\label{section_concluding_remarks}
\setcounter{equation}{0}

Can more examples of exponential Dowling structures be given? For
instance, find the Dowling extension of counting matrices with
non-negative integer entries having a fixed row and column sum.
See~\cite[Chapter~5]{Stanley_EC_II}.

Theorem~\ref{theorem_m_plus_1}
has been generalized
in~\cite{Ehrenborg_Readdy_restricted}.
As we have seen in
this theorem
the generating function
for the M\"obius function
of $D_{n}^{(r,k)} \cup \{\hz\}$
in
Proposition~\ref{proposition_M}
in the case when the order $s$ is equal to $1$
has a permutation enumeration analogue.
It would be interesting to find a permutation
interpretation for this generating function
for general values of the order $s$.
Similar generating functions
have appeared when enumerating classes
of $r$-signed permutations.
A few examples are
$(\sin(px) + \cos((r-p)x)/\cos(rx)$
counting
$p$-augmented $r$-signed permutations
in~\cite{Ehrenborg_Readdy_Sheffer},
$\sqrt[r]{1/(1 - \sin(rx))}$
counting
augmented Andr\'e $r$-signed permutations
in~\cite{Ehrenborg_Readdy_r-cubical},
and
$\sqrt[r]{1/(1 - rx)}$
counting
$r$-multipermutations
in~\cite{Park}.

There are several other questions to raise.
Is there a $q$-analogue of the partition lattice
such that a natural $q$-analogue of
Theorem~\ref{theorem_m_plus_1} also holds?
We only use the case $w = \av^{k-1}$ in
Proposition~\ref{proposition_permutations}.
Are there other poset statistics that correspond to
other $\ab$-words $w$?

The symmetric group $\Ssss_{m-1}$ acts naturally on
the lattice $\Pi_{m}^{r,j}$. Hence it also acts
on the top homology group of
$\Pi_{m}^{r,j}$. 
In a forthcoming paper we study
the representation of this $\Ssss_{m-1}$ action.

Similar questions arise concerning the poset
$D_{n}^{(r,k)} \cup \{\hz\}$;
see Proposition~\ref{proposition_M}.
Is this poset shellable?
Is the homology of
this poset
concentrated in the top homology?
Note that the wreath product
$G \wr \Ssss_{n}$ acts on
the Dowling lattice $L_{n}(G) = L_{n}$.
Hence $G \wr \Ssss_{n}$ acts on
the exponential Dowling structure
$D_{n}^{(r,k)} \cup \{\hz\}$.
What can be said about the action of
the wreath product
$G \wr \Ssss_{n}$
on
the homology group(s) of
$D_{n}^{(r,k)} \cup \{\hz\}$?

\section*{Acknowledgements}

The first author was partially supported by National Science
Foundation grant 0200624. Both authors thank the
Mittag-Leffler Institute where a portion of this research was completed
during the Spring 2005 program in Algebraic Combinatorics.
The authors also thank the referee for suggesting
additional references.


\newcommand{\journal}[6]{{\sc #1,} #2, {\it #3} {\bf #4} (#5), #6.}
\newcommand{\book}[4]{{\sc #1,} ``#2,'' #3, #4.}
\newcommand{\bookf}[5]{{\sc #1,} ``#2,'' #3, #4, #5.}
\newcommand{\thesis}[4]{{\sc #1,} ``#2,'' Doctoral dissertation, #3, #4.}
\newcommand{\springer}[4]{{\sc #1,} ``#2,'' Lecture Notes in Math.,
                          Vol.\ #3, Springer-Verlag, Berlin, #4.}
\newcommand{\preprint}[3]{{\sc #1,} #2, preprint #3.}
\newcommand{\preparation}[2]{{\sc #1,} #2, in preparation.}
\newcommand{\appear}[3]{{\sc #1,} #2, to appear in {\it #3}}
\newcommand{\submitted}[4]{{\sc #1,} #2, submitted to {\it #3}, #4.}
\newcommand{\JCTA}{J.\ Combin.\ Theory Ser.\ A}
\newcommand{\JCTB}{J.\ Combin.\ Theory Ser.\ B}
\newcommand{\AdvancesinMathematics}{Adv.\ Math.}
\newcommand{\JournalofAlgebraicCombinatorics}{J.\ Algebraic Combin.}

\newcommand{\communication}[1]{{\sc #1,} personal communication.}


\vspace*{10 mm}

{\em
\noindent
R.\ Ehrenborg,
Department of Mathematics,
University of Kentucky,
Lexington, KY 40506 \newline
\noindent
M.\ Readdy,
Department of Mathematics,
University of Kentucky,
Lexington, KY 40506 \newline
{\tt jrge@ms.uky.edu},
{\tt readdy@ms.uky.edu}
}

\end{document}